\documentclass[a4paper,11pt]{amsart}

\usepackage[utf8]{inputenc}

\usepackage{amssymb}
\usepackage{latexsym}
\usepackage{amsthm}

\usepackage{hyperref}
\usepackage{multirow}
\usepackage{enumerate}
\usepackage{mathrsfs}
\usepackage{mathtools}

\usepackage[all]{xy}

\newtheorem{thm}{Theorem}[section]

\newtheorem{cor}[thm]{Corollary}
\newtheorem{lem}[thm]{Lemma}

\newtheorem{prop}[thm]{Proposition}

\theoremstyle{definition}

\theoremstyle{definition}
\newtheorem{defn}[thm]{Definition}

\newenvironment{pf}{\par\noindent{\bf Proof.}\enspace\ignorespaces}{\qed\par\par}

\newcommand{\cU}{{\mathcal{U}}}
\newcommand{\cC}{{\mathcal{C}}}

\newcommand{\cP}{{\mathcal{P}}}

\newcommand{\cF}{{\mathcal{F}}}
\newcommand{\cG}{{\mathcal{G}}}

\newcommand{\hcF}{{\widehat{\cF}}}
\newcommand{\hcG}{{\widehat{\cG}}}

\DeclareMathOperator{\id}{id}
\DeclareMathOperator{\eq}{eq}
\DeclareMathOperator{\coeq}{coeq}

\DeclareMathOperator{\Sets}{\mathbf{Sets}}
\DeclareMathOperator{\Top}{\mathbf{Top}}
\DeclareMathOperator{\CH}{\mathbf{CHaus}}

\DeclareMathOperator{\bSets}{\beta \Sets}

\DeclareMathOperator{\cSets}{\mathbf{CSets}}
\DeclareMathOperator{\ncSets}{\mathbf{NCSets}}

\newcommand{\wbSets}{{\widehat{ \bSets}}}
\newcommand{\wCH}{{\widehat{\CH}}}

\DeclareMathOperator{\res}{\mathbf{Res}}
\DeclareMathOperator{\ex}{\mathbf{Ext}}

\DeclareMathOperator{\colim}{colim}

\newcommand{\BB}{B^{(2)}}

\begin{document}

	\title{Condensed Sets via free resolutions}

	\author[Damià Rodríguez Banús]{Damià Rodríguez Banús}
	\email{damia@abox.com}
	
	\author[Xavier Xarles]{Xavier Xarles}
	\address{Departament de Matem\`atiques\\Universitat Aut\`onoma de
		Barcelona\\08193 Bellaterra, Barcelona, Catalonia}
	\email{xarles@mat.uab.cat}

	\begin{abstract}
	    We show the equivalence of several constructions of the category of condensed sets by using free resolutions of compact Hausdorff spaces. We also give an elementary construction of the condensed set associated to any presheaf on compact Hausdorff spaces. 
	\end{abstract}

	\maketitle
	
	The theory of condensed sets, recently developed by Dustin Clausen	and Peter Scholze \cite{CS} (see also the related construction of Pyknotic objects by Clark Barwick and Peter Haine \cite{BH}), claims that a \textit{nice} category of topological spaces, including compact Hausdorff spaces and verifying some desirable properties, should be replaced by certain objects defined in a functor category, concretely a category of sheaves with respect to certain Grothendieck topology on profinite spaces. One of their first results is showing that in their definition one can either use compact Hausdorff spaces, profinite spaces, or even extremally disconnected topological spaces, and one gets equivalent categories. In this note we work out the details of a precise comparison between the category of condensed sets defined as certain presheaves on free compact Hausdorff spaces, i.e. $\beta$-sets, and the category of condensed sets defined as sheaves on compact Hausdorff spaces, which easily implies that the category of condensed sets can be constructed using any of the various subcategories of topological spaces cited previously. This idea is already indicated in the cited notes.

   Although we are already aware that the main result can be proved as application of a general theory of Grothendieck topologies, by using that the $\beta$-sets form a basis for the topology of finite jointly epimorphic families for the compact Hausdorff spaces (see, for example, Proposition B.6.6 of \cite{Lur}, or \cite{Stacks}, Tag 03A1), we think that an elementary proof of this result can be useful for the people approaching condensed sets. Moreover, this note can also be used as an elementary introduction to the notion of condensed sets, and includes also some useful results about $\beta$-sets.

    As a by-product of our result, we also get another way, maybe more elementary, to construct the condensed set associated to any presheaf on compact Hausdorff spaces. 
    
    \textbf{Acknowledgements.} We thank Francesc Bars, Natalia Castellana, Marc Masdeu, Enric Nart, Joaquim Roé and Eloi Torrents, as participants in the tropical seminar in the U.A.B., for comments and discussions about this topic. We thank Marc Masdeu also for some corrections on a first version of the paper.  We thank Florian Leptien for pointing out an error in a previous version of the paper, and specially for suggesting a proof of Proposition \ref{extensionverifiesstar}. 
    The second author is partially supported by grant PID2020-
116542GB-I00 from the Spanish Research Agency.

\section*{Preliminaries and notations}

We will consider an essentially small category $\Sets$ of sets such that, if $S$ is in $\Sets$, so it is the set of its parts $\cP(S)$. For example, if $\kappa$ is an uncountable strong limit cardinal, we can consider the category of sets with cardinality strictly smaller than $\kappa$ (see remark 1.3 in \cite{CS}). All the other categories, such as the categories of topological spaces, $\Top$, and compact Hausdorff spaces, $\CH$, will be considered with underlying sets in $\Sets$. 

Given a topological space $X$, we will denote by $|X|$ the underlying set, and equally $|f|:|X|\to |Y|$ for any $f:X\to Y$ continuous map. We will also consider $j_X:|X|\to X$ the (continuous) map given by the identity as sets. 

Recall that the forgetful functor $|\ |:\CH\to \Sets$ commutes with fibre products and with coproducts. Therefore, $|X \times_Z Y|=|X| \times_{|Z|} |Y|$ for any two morphisms $f:X\to Z$ and $g:Y\to Z$ in $\CH$, and, similarly, $|X\sqcup Y|=|X|\sqcup |Y|$. 

\section{$\beta$ Sets}

In this section, we introduce the category $\bSets$ as the category of free objects in the category $\CH$, and we show some of its properties. 

Given a set $S$ in $\Sets$, let $\beta S$ be the set of ultrafilters on $S$, with the natural topology that makes it a compact Hausdorff space (see \cite{compactum}). Every function $f:S\to T$ induces a continuous function $\beta f:\beta S\to \beta T$ using the usual application of functions to filters, so we have a functor $\beta:\Sets \to \CH$. We denote by $ \bSets $ the full subcategory of $\CH$ formed by the objects $\beta S$.

Given any set $S$ in $\Sets$, consider the map $\iota_X:S\to \beta S$ sending any element to the principal ultrafilter. Then $\iota_X$ is injective, and it has dense image. Moreover, for any map $f:S\to T$, the natural diagram 
$$
\xymatrix{
	S \ar@{^{(}->}[r]^-{\iota_S} \ar[d]_{f} & \beta S \ar[d]^{\beta f} \\
	T \ar@{^{(}->}[r]_-{\iota_T} & \beta T
}
$$
commutes. 

It is well known that for any $S$, the space $\beta S$ is the Stone-\v Cech compactification of $S$, considered as a topological space with the compact topology. Hence, for any compact Hausdorff space $X$ and any (continuous) map $f: S\to X$, there is a unique continuous map $\tilde{f}: \beta S\to X$ such that $\tilde{f}\circ \iota_S=f$. 

This map can be constructed in the following way: given a compact Hausdorff space $X$, consider the underlying set $|X|$ and the continuous map $j_X:|X|\to X$ which is the identity as a map between sets, where $|X|$ is considered with the discrete topology. Then there is a continuous map $\xi_X:\beta|X|\to X$ given by sending any ultrafilter of $X$ to its limit in $X$, and such that $\xi_X\circ \iota_X=j_X$. Then $\tilde{f}=\xi_X\circ \beta f$.

This map is used to show that $\beta$ is the left adjoin of the forgetful functor $|\ |:\CH\to \Sets$. Hence, $\beta S$ is the free object associated to a set $S$, and therefore the category $\bSets$ is the category of free compact Hausdorff spaces. 

Recall some of the basic properties of these maps. 

\begin{lem}\label{Bxiandiota} Let $f:X\to Y$ a morphism in $\CH$. Then 
	$f\circ \xi_X=\xi_Y\circ \beta f$ and 
	$\beta f \circ \iota_X=\iota_Y\circ |f|$. 
\end{lem}
\begin{proof}
    First, the fact that $\beta f \circ \iota_X=\iota_Y\circ |f|$ is directly deduced from the definition of $\beta f$. On the other hand, consider the following diagram
    $$
    \xymatrix{
        |X| \ar@{^{(}->}[r]^-{\iota_X} \ar[d]_{|f|} & \beta ( |X| ) \ar[r]^-{{\xi_X}} \ar[d]_-{{\beta f}} & X \ar[d]^-{f} \\
        |Y| \ar@{^{(}->}[r]^-{\iota_Y} & \beta ( |Y| ) \ar[r]^-{{\xi_Y}} & Y 
    }
    $$
    the left square of which commutes. Then, we have the following equalities of maps in $\mathbf{Sets}$:
    $$
    (f \circ \xi_X) \circ i_X = f, \quad (\xi_Y \circ \beta f) \circ i_X = f,
    $$
    given that $\xi \circ i = id$ i $\beta f \circ i_X = i_Y \circ |f|$. Hence, the diagram
    $$
    \xymatrix{
    |X| \ar@{^{(}->}[r]^-{i_X} \ar[dr]_-{f} & \beta ( |X| ) \ar@<-1ex>[d] \ar@<1ex>[d] \\
    	& Y 
    }
    $$
    commutes and, by the universal property of the Stone-\v Cech compactification, we have that $f \circ \xi_X = \xi_Y \circ \beta f$.
\end{proof}

The following lemma is quite useful for showing that a map from a $\beta$-set to a compact Hausdorff space is an epimorphism, and it can be deduced from the facts that $\iota_S:S\to \beta S$ has dense image, and the well-know result that the epimorphisms in the category of compact and Hausdorff spaces are the continuous maps with dense image. Nevertheless, we have decided to include a direct proof of this lemma.

\begin{lem}\label{epibetaCH}
	Let $f: \beta S \to X$ be any map from a $\beta$-set $\beta S$ to a compact Hausdorff space. Then $f$ is an epimorphism if and only if $f \circ \iota_S$ has dense image. 
\end{lem}
\begin{proof}
         Suppose first that $f$ is an epimorphism. Consider $x \in X$ and $N$ a neighborhood of $x$ such that $x \in U \subseteq N \subseteq X$, where $U$ is an open set. Consider the open set $V := f^{-1}(U) \subseteq \beta S$ and $F \in \beta S$ such that $f(F)=x$. Given that $i_S (S)$ is dense in $\beta S$ and $V \neq \emptyset$ because $F \in V$, there is $z \in S$ with $i_S (z) \in V$. Thus, $f(i_S (z)) \in N \cap f(i_S (S)) \neq \emptyset$ and $x \in \overline{f(i_S (S))}$. Hence, $\overline{f(i_S (S))} = X$.
         
        Assume now that $f \circ i_S$ has dense image. Firstly,
    	$$
    	f(i_S (S)) \subseteq f(\beta S) \subseteq X.
    	$$
    	From the fact that $f(\beta S)$ is compact because $f$ is a continuous map, and it is a closed set because $X$ is Hausdorff, we deduce that
    	$$
    	X = \overline{f(i_S (S))} \subseteq f(\beta S) \subseteq X.
    	$$
    	Consequently, $f$ is surjective and so it is an epimorphism.
\end{proof}

\begin{prop}\label{BSetsCat}
	The category $ \bSets$ has initial and final objects and finite coproducts. The functor $\beta$ preserves them. Moreover, the coproducts in $ \bSets$ are the coproducts in $\CH$ (or in $\Top$), and hence the underlying sets of the coproducts are the coproducts of the underlying sets.
\end{prop}
\begin{proof}
    Firstly, the Stone-\v Cech compactification of the empty set is the empty set itself. Hence, for every $\beta$-set $\beta S$ there is a unique morphism $\emptyset \to \beta S$, thus $\bSets$ has initial object. Secondly, the Stone-\v Cech compactification of $\lbrace \ast \rbrace$ is $\beta \left\{ \ast \right\} = \left\{ \ast \right\}$, thus for every $\beta$-set $\beta S$ there is a unique morphism $\beta S \to \beta \left\{ \ast \right\}$, thus $\bSets$ has final object. 
    
    On the other hand, consider $\beta X$ and $\beta Y$ two  $\beta$-sets. Let $S: \mathbf{BooRng} \to \mathbf{Stone}^{\textrm{op}}$ be the functor given by the Stone Duality, assigning to each boolean ring $R$ the space $\textrm{Spec}(R)$ equipped with the Zariski topology. It is easy to check that $\beta X \cong S(\mathcal{P}(X))$ as topological spaces. As a result, $\beta (X \sqcup Y) \cong S(\mathcal{P}(X \sqcup Y))$. Finally, from the fact that $\mathcal{P}(X \sqcup Y) \cong \mathcal{P}(X) \times \mathcal{P}(Y)$ and $S$ is an (anti)equivalence (therefore preserves limits), we deduce that $S(\mathcal{P}(X \sqcup Y))$ is the coproduct of $S(\mathcal{P}(X))$ and $S(\mathcal{P}(Y))$. In particular, $\beta (X \sqcup Y)$ is the coproduct of $\beta X$ and $\beta Y$ and coincides with the coproduct in $\mathbf{Top}$ (i.e. $\beta (X \sqcup Y) \cong \beta X \sqcup \beta Y$).  
\end{proof}

The following lemma shows that, although there are much more morphisms as $\bSets$ than as $\Sets$, two of them are isomorphic if and only if they are in bijection. 

\begin{lem}\label{isobeta}
    Let $S$ and $T$ be sets in $\Sets$. Suppose that $\beta S$ and $\beta T$ are homeomorphic as topological spaces. Then there is a bijection between $S$ and $T$. 
\end{lem}
\begin{pf}
    Using the Stone duality, an homemorphism between $\beta S$ and $\beta T$ gives an isomorphism as Boolean algebras between $\cP(S)$ and $\cP(T)$. But any such isomorphism gives a bijection between the respective atoms, and the atoms of $\cP(S)$ correspond to the subsets $\{s\}$ for $s \in S$. 
\end{pf}

Next, we show that the $\beta$-sets are projective elements in the category of compact Hausdorff spaces, and so that $\beta$-sets are indeed extremally disconnected. 

\begin{lem}\label{proj}\label{betaex}
	
	Let $f:\beta X\to Y$ be a continuos map where $Y$ is compact Hausdorff space, and let $g: Z\to Y$ and epimorphism of compact Hausdorff spaces. Then there exists $h: \beta X\to Z$ such that $g\circ h=f$. 
	
	In particular, if $g: Y \to \beta X$ is an epimorphism of compact Hausdorff spaces, there exists a continuous map $h:\beta X \to Y$ such that $f\circ s=\id_{\beta_X}$. 
\end{lem}

\begin{pf}
	Firstly, since $g$ is surjective, there exists a section map $s: |Y| \to Z$ as sets, with $g\circ s=\id_|Y|$. The composition $t:=s\circ f\circ \iota_X: X\to Z$ is a continuous map, because $X$ is a discrete space. Therefore, by the universal property of the Stone-\v Cech compactification, there is a continuous map $h: \beta X \to Z$ such that $ t = h \circ \iota_X$. Now, observe that
	$$
	g\circ h \circ \iota_X = g \circ (s\circ f\circ \iota_X) = f\circ \iota_X. 
	$$
	Therefore, we have that $g\circ h= f$ since there is a unique continuous map $f:\beta X\to Y$ extending a given map $X\to Y$.
\end{pf}

\begin{prop}\label{betafiberproduct}
Let  $T_1$, $T_2$ and $S$ be in $\Sets$. Let $f_1: \beta T_1 \to \beta S$ and $f_2: \beta T_2 \to \beta S$ be continuous maps. Consider the fibre product $\beta T_1 \times_{ \beta S}  \beta T_2$ in the category $\CH$, together with the projection maps $p_i: \beta T_1 \times_{\beta S} \beta T_2 \to \beta T_i$
for $i=1,2$.

Then the $\beta$-set $\beta|\beta T_1 \times_{\beta S} \beta T_2|$, together with the natural maps \[\pi_i: \beta|\beta T_1 \times_{\beta S} \beta T_2|\to \beta T_i \text{ for } i=1,2,\] given as $\pi_i=p_i\circ \xi_{\beta T_1 \times_{\beta S} \beta T_2}$, 
 verifies the following existence property: for any set $Q$ and a commutative diagram 
    $$
    \xymatrix{
    {\beta Q} \ar[r]^-{q_1} \ar[d]_-{q_2} & {\beta T_1} \ar[d]^-{f_1} \\
    {\beta T_2} \ar[r]_-{f_2} & {\beta S}
    }
    $$
there exists a morphism $\tau: \beta Q  \to \beta|\beta T_1 \times_{\beta S} \beta T_2|$ such that $\pi_i\circ \tau=q_i$ for $i=1$ and $2$. 
\end{prop}
\begin{proof}
    First, because of the universal property of the fibre product $\beta T_1 \times_{ \beta S}  \beta T_2$, there is a unique continuous map $h: \beta Q \to \beta T_1 \times_{ \beta S}  \beta T_2$ such that the following diagram commutes:
    $$
    \xymatrix{
    \beta Q \ar@/_/[ddr]_{q_2} \ar@/^/[drr]^{q_1} \ar[dr]^{h} \\
    & \beta T_1 \times_{ \beta S}  \beta T_2 \ar[r]_-{p_1} \ar[d]_-{p_2} & \beta T_1 \ar[d]^-{f_1} \\
    & \beta T_2 \ar[r]_-{f_2} & \beta S
    }
    $$
    Now, consider the following diagram in $\mathbf{Sets}$:
    $$
    \xymatrix{
    Q \ar@{^{(}->}[r]^-{\iota_Q} \ar[d]_-{{h \circ \iota_Q}} & {\beta Q} \ar[d]_-{{\beta(h \circ \iota_Q)}} \ar[r]^-{h} & {|\beta T_1 \times_{ \beta S}  \beta T_2|} \\
    {|\beta T_1 \times_{ \beta S}  \beta T_2|} \ar@{^{(}->}[r]^-{\iota} & {\beta | \beta T_1 \times_{ \beta S}  \beta T_2 |} \ar[ur]_-{\xi} 
    }
    $$
    The left square commutes, and we would like the right triangle to commute too. Actually, notice that, as maps in $\mathbf{Sets}$, we have the following equalities:
    $$
    \xi \circ \beta (h \circ \iota_Q) \circ i_Q = \xi \circ \iota \circ h \circ i_Q = id_{|\beta T_1 \times_{ \beta S}  \beta T_2|} \circ h \circ i_Q = h \circ i_Q
    $$
    Therefore, as continuous maps we have that $\xi \circ \beta (h \circ \iota_Q) \circ i_Q = h \circ i_Q$. As a result of the universal property of $\beta Q$, the equality $\xi \circ \beta (h \circ \iota_Q) = h$ is satisfied. Thus, the morphism $\tau=\beta (h \circ \iota_Q) : \beta Q \to \beta | \beta T_1 \times_{ \beta S}  \beta T_2 |$ satisfies the desired property.
\end{proof}

\section{Free resolution of compact Hausdorff topological spaces}

Given a compact Hausdorff space $X$, consider the underlying set $|X|$ and the continuous map $j_X:|X|\to X$, which is the identity as a map between sets, where the set $|X|$ is equipped with the discrete topology. Now, let $B(X):=\beta |X|$ be the $\beta$-set associated with $|X|$, and consider the continuous map $\xi_X:B(X)\to X$. The map $\xi_X$ is a surjection between compact Hausdorff spaces, hence it is a quotient map too. Moreover, it is an epimorphism in the category of compact Hausdorff spaces.

Now, consider the fibre product in the category of topological spaces $B(X)\times_XB(X)$ using the map $\xi_X$ for both factors. Consequently, we have two projection maps $p_i:B(X)\times_XB(X)\to B(X)$ for $i=1,2$ for which $X$ becomes (isomorphic to) its coequalizer. As $B(X)\times_X B(X)$ is not a $\beta$-set, we take $B^{(2)}(X):=B(B(X)\times_X B(X))$, and we get the maps $\pi_i:= p_i \circ \xi_{B(X)\times_X B(X)}$ for $i=1,2$.

\begin{lem}\label{freeresolution}
	The map $\xi_X:B(X)\to X$ gives an isomorphism 
	\[\coeq\left(
		B^{(2)}(X)\underset{\pi_2}{\overset{\pi_1}{\rightrightarrows}} B(X)
	\right) \cong X\]
We will call this the standard free resolution of $X$.
\end{lem}
\begin{pf}
    It suffices to show that $\xi_X:B(X)\to X$ is a coequalizer of 
    $$
    B(X) \times_X B(X) \underset{p_2}{\overset{p_1}{\rightrightarrows}}  B(X),
    $$
    since the map $$\xi_{B(X)\times_X B(X)}: B^{(2)}(X)\to B(X)\times_XB(X)$$ is an epimorphism. 
    Firstly, by construction of the fibre product, we have that $\xi_X \circ p_1 = \xi_X \circ p_2$. Suppose there is a morphism $c: B(X) \to Y$ in $\CH$ such that $c \circ p_1 = c \circ p_2$ and define $k:= c \circ \iota_X$. Then, given that $k \circ \xi_X \circ \iota_X = c \circ \iota_X$, due to the uniqueness of the universal property of the Stone-\v Cech compactification, we have that $k \circ \xi_X = c$. Finally, $k$ is the unique morphism satisfying $k \circ \xi_X = c$. If there were a morphism $k'$ such that $k' \circ \xi_X = c$, then we would have $k \circ \xi_X = k' \circ \xi_X$ and consequently $k = k'$ because $\xi_X$ is an epimorphism.
\end{pf}

During the rest of the section, we will prove some properties of this resolution.

\begin{prop}\label{Bprop}
	The functor $B:\CH\to  \bSets$ preserve epimorphisms and finite coproducts.  
\end{prop}

\begin{pf}
	Let $f:X\to Y$ be an epimorphism in $\CH$, hence a continuous surjective map between compact Hausdorff spaces. By using lemma \ref{epibetaCH}, in order to show that the map $B(f):B(X)\to B(Y)$ is an epimorphism we only need to show that $B(f)\circ \iota_X:|X|\to B(Y)$ has dense image. But $B(f)\circ \iota_X=\iota_Y\circ |f|$ by lemma \ref{Bxiandiota}, $|f|$ is surjective and $\iota_Y$ has dense image.
	
Finally, the last property is clear as the functor $B$ is the composition of the forgetful functor $X\mapsto |X|$, which preserve finite coproducts, with the $\beta$ functor, which preserve finite coproducts by proposition \ref{BSetsCat}. 	
\end{pf}

\begin{prop}\label{B2prop} The assignment $X\mapsto B^{(2)}(X)$ determines a functor $B^{(2)}:\CH\to  \bSets$ which preserve finite coproducts. 
\end{prop}
\begin{pf}
	For any compact Hausdorff space $X$, we will denote by $\widetilde{B}(X)$ the fibre product $B(X)\times_X B(X)$, so $B^{(2)}(X)=B(\widetilde{B}(X))$. Once we show that $\widetilde{B}$ determines a functor with all the properties we will be done by using the previous proposition. 
		
	Given any map $f:Y\to X$ in $\CH$, consider the unique map $\widetilde{B}(f)$ such that  
	$$
	\xymatrix{
		\widetilde{B}(Y)=B(Y) \times_Y B(Y) \ar@<1ex>[r] \ar@<-1ex>[r] \ar[d]_-{\widetilde{B}(f)} & B(Y) \ar[r]^-{\xi_Y} \ar[d]_-{\beta f} & Y \ar[d]^-{f} \\
		\widetilde{B}(X)=B(X) \times_X B(X) \ar@<1ex>[r] \ar@<-1ex>[r] & B(X) \ar[r]^-{\xi_X} & X 
	}
	$$
	given by the universal property of the fibre product, where we used the commutativity $f\circ \xi_Y=\xi_X\circ \beta f$ from lemma \ref{Bxiandiota}. One verifies that this gives the functor  $\widetilde{B}$.

Now, we need to prove that the functor $\tilde{B}$ preserves finite coproducts. It is easy to check that it preserves the initial object. For non-empty coproducts, it will suffice to show that
$$
B(X \sqcup Y) \times_{X \sqcup Y} B(X \sqcup Y) = \left( B(X) \times_X B(X) \right) \sqcup \left( B(Y) \times_Y B(Y) \right),
$$
But we have that 
\begin{multline*}
\left( B(X) \sqcup B(Y) \right) \times_{X \sqcup Y} \left( B(X) \sqcup B(Y) \right) = 
\\ = (B(X) \times_{X \sqcup Y} B(X)) \sqcup (B(X) \times_{X \sqcup Y} B(Y)) \sqcup \\ \sqcup  (B(Y) \times_{X \sqcup Y} B(X)) \sqcup (B(Y) \times_{X \sqcup Y} B(Y))
\end{multline*}
and it is clear that 
\[B(X) \times_{X \sqcup Y} B(X) = B(X) \times_{X} B(X) 
 , \quad B(Y) \times_{X \sqcup Y} B(Y) = B(Y) \times_{Y} B(Y) \]
\[
B(X) \times_{X \sqcup Y} B(Y) = \emptyset = B(Y) \times_{X \sqcup Y} B(X). 
\]
\end{pf}

\begin{cor}\label{mapfreeresolutions}
	Given any map $f:X\to Y$ in $\CH$, the following diagram commutes 
	$$
	\xymatrix{
		{\BB (X)} \ar[d]_-{{\BB f}} \ar@<1ex>[r] \ar@<-1ex>[r] & B(X) \ar[d]_-{\beta f} \ar[r] & X \ar[d]_-{f} \\
		{\BB (Y)} \ar@<1ex>[r] \ar@<-1ex>[r] & B(Y) \ar[r] & Y
	}
	$$	

\end{cor}
\begin{pf}
    Taking into account the previous discussion about the functor $\tilde{B}$, the commutativity of the diagram is guaranteed by lemma \ref{Bxiandiota}. 
\end{pf}

\begin{cor}\label{coproductfreeresolutions}
	Given two objects $X$ and $Y$ in $\CH$. Then 	$$
	\xymatrix{
		{\BB (X) \sqcup \BB (Y)} \ar@<1ex>[r]^-{{\tilde{p_1} \sqcup \tilde{q_1}}}  \ar@<-1ex>[r]_-{{\tilde{p_2} \sqcup \tilde{q_2}}} & {B(X) \sqcup B(Y)} \ar[r]_-{{\xi_X \sqcup \xi_Y}} & {X \sqcup Y} 
	}
	$$
	is the standard free resolution of $X\sqcup Y$. 	
\end{cor}
\begin{pf}
    First, consider the standard free resolutions of $X$ and $Y$
    $$
    \xymatrix{
    {\BB (X)} \ar@<1ex>[r]^-{{\tilde{p_1}}} \ar@<-1ex>[r]_-{{\tilde{p_2}}} & B(X) \ar[r]^-{{\xi_X}} & X & {\BB (Y)} \ar@<1ex>[r]^-{{\tilde{q_1}}} \ar@<-1ex>[r]_-{{\tilde{q_2}}} & B(Y) \ar[r]^-{{\xi_Y}} & Y
    }
    $$
    Additionally, notice that the argument used in proposition \ref{B2prop} to prove that coproducts are preserved is equally valid for showing that the maps $\xi_{X \sqcup Y} : B(X \sqcup Y) \to X \sqcup Y$ and $\xi_X \sqcup \xi_Y : B(X) \sqcup B(Y) \to X \sqcup Y$ are ``equal'' (considering that $B(X \sqcup Y) \cong B(X) \sqcup B(Y)$). Therefore, we can immediately deduce that
    $$
	\xymatrix{
		{\BB (X) \sqcup \BB (Y)} \ar@<1ex>[r]^-{{\tilde{p_1} \sqcup \tilde{q_1}}}  \ar@<-1ex>[r]_-{{\tilde{p_2} \sqcup \tilde{q_2}}} & {B(X) \sqcup B(Y)} \ar[r]_-{{\xi_X \sqcup \xi_Y}} & {X \sqcup Y} 
	}
	$$
	is the standard free resolution of $X\sqcup Y$.
\end{pf}

\section{Naive Condensed Sets.}

We will denote by $\wCH$ and $\wbSets$ the categories of contravariant functors $\CH\to \Sets$ and $\bSets\to \Sets$, respectively, also called the categories of presheaves, with natural transformations as maps. 

We will denote by $\wCH_{\star}$ the full subcategory of functors $\cF$ in $\wCH$ such that, for any epimorphism $f:Y\to X$ in $\CH$, the natural map 
\begin{equation} \label{star}
    \cF(X)\to \eq (\cF(Y)\rightrightarrows  \cF(Y\times_XY) )
\end{equation}
is a bijection.  We will say that $\cF$ verifies property $(\star)$. 

We will denote by $\wbSets_{\times}$ the full subcategory of $\wbSets$ that send finite coproducts to finite products; equivalently, they are the contravariant functors $\cF:\bSets\to \Sets$ such that $\cF(\emptyset)$ is a one element set and that for any sets $S$ and $T$, we have $\cF(\beta S\sqcup \beta T)\cong \cF(\beta S)\times \cF(\beta T)$ via the natural map.

\begin{defn}
	A naive condensed set $\cC$ is an element of $\wbSets_{\times}$.
	
	The category $\ncSets$ of naive condensed sets is $\wbSets_{\times}$.
\end{defn}

Recall the following definition of Clausen and Scholze of the category of condensed sets.

\begin{defn}
	A condensed set $\cF$ is a contravariant functor from $\CH$ to $\Sets$ which sends finite coproducts to finite products, and such that, for any epimorphism $f:Y\to X$, the natural map 
	$$\cF(X)\to \eq (\cF(Y)\rightrightarrows \cF(Y\times_XY) )$$
	is a bijection. 
	
	Hence, the category $\cSets$ of condensed sets is the full subcategory of the functor category $\wCH_{\star}$ that sends finite coproducts to finite products.
\end{defn}

\begin{thm}\label{equivalence}
	The restriction functor $\res: \wCH \to \wbSets$ gives an equivalence of categories between $\wCH_{\star}$ and  $\wbSets$.
\end{thm}

\begin{cor}\label{equivalencestar}
	The restriction functor $\res:\wCH \to \wbSets$ gives an equivalence of categories between $\cSets$ and  $\ncSets$.
\end{cor}

To prove the theorem and the corollary we will construct an extension functor 
$\ex: \wbSets\to \wCH$ which will give us the inverse functor in both cases. 

Given any $\cF \in \wbSets$, and given any $X$ in $\CH$, we consider 
$$ \BB(X)) \rightrightarrows B (X) \to X $$
the standard free resolution of $X$ and we define 
$$
\hcF(X):= \eq \left( \cF(B(X)) \rightrightarrows \cF(\BB (X)) \right)
$$

\begin{lem}\label{extension}
	Given any $\cF \in \wbSets$, the assignment $X\mapsto \hcF(X)$ gives an element $\hcF\in \wCH$. 
	
	The assignment $ \cF\mapsto \hcF$ determines a functor $\ex:\wbSets\to \wCH$, which we call the extension functor from $\wbSets$ to $\wCH$. 
\end{lem}

\begin{pf}
	Firstly, we need to show that for any $\cF \in \wbSets$, the assignment $\hcF$ can be extended to a functor. But this is clear by using that given any map $f:X\to Y$ in $\CH$, we have a map of their respective free resolutions by Corollary \ref{mapfreeresolutions}, hence we can define $\hcF(f):\hcF(Y)\to \hcF(X)$. One easily shows that this construction is compatible with the identity maps and the composition. 
	
	Now, in order to define the extension functor $\ex$, we need to show that any natural transformation $\tau:\cF\to \cG$ between functors $\cF, \cG$ in $\wbSets$ gives rise to a natural transformation $\widehat{\tau}:\hcF\to \hcG$. For any $X$ in $\CH$, the map $ \widehat{\tau}_X:\hcF(X)\to \hcG(X)$ is the map given by the equalizer's universal property, applied to the diagram 	
	\[
	\xymatrix{
		{\hcF(X)} \ar[r] \ar@{.>}[d]_-{ \widehat{\tau}_X} & \cF(B(X)) \ar[d]_-{\tau_{B(X)}} \ar@<1ex>[r] \ar@<-1ex>[r] & {\cF(\BB(X))} \ar[d]_-{{\tau_{\BB (X)}}}   \\
		{\hcG(X)} \ar[r] & \cG(B(X)) \ar@<1ex>[r] \ar@<-1ex>[r] & {\cG(\BB(X))} } 
     \]
But, for any map $f:X\to Y$ in $\CH$, it is straightforward to show that $\hcG(f) \circ \widehat{\tau}_X = \widehat{\tau}_Y \circ \hcF(f)$, which proves that $\widehat{\tau}$ is indeed a natural transformation. 	
\end{pf}

In order to prove the main theorem, we will make use of the following technical key lemma.

\begin{lem}\label{clau}
	Let $\cC$ be any category and $A, B, C \in \textrm{Ob}(\mathbf{C})$ together with morphisms 
	\[
	\xymatrix@1{ C \ar@<1ex>[r]^-{p_1} \ar@<-1ex>[r]_-{p_2} & B \ar@<1ex>[r]^-{f} & A \ar@<1ex>[l]^-{g}}
	\] such that $f \circ p_1 = f \circ p_2$ and $f \circ g = id_A$. Suppose  that there exists a morphism $k: B \to C$ such that the following diagram
	$$
	\xymatrix{
		B \ar@/_/[ddr]_{{g \circ f}} \ar@/^/[drr]^{id_B} \ar[dr]^k \\
		& C \ar[d]_-{p_2} \ar[r]^-{p_1} & B \ar[d]^-{f} \\
		& B \ar[r]_-{f} & A	
	}
	$$
	commutes. Then, for any contravariant functor $F: \cC \to \Sets$ one has that 
	$$
	\xymatrix@1{
		F(A) \ar[r]^-{F(f)} & F(B) \ar@<1ex>[r]^-{F(p_1)} \ar@<-1ex>[r]_-{F(p_2)} & F(C)
	}
	$$
	is an equalizer. 

\end{lem}
\begin{pf}
	Denote by $i_E : E \hookrightarrow F(B)$ the equalizer of
	$F(p_i): F(B) \rightrightarrows F(A)$ for $i=1,2$. 
	 As $F(p_1) \circ F(f) = F(p_2) \circ F(f)$, we have a unique morphism $\hat{f}: F(A) \to E$ such that  
	$$
	\xymatrix{
		F(A) \ar[d]_-{{\hat{f}}} \ar[dr]^-{F(f)} \\
		E \ar@{^{(}->}[r]_-{i_E} & F(B) \ar@<1ex>[r]^-{F(p_1)} \ar@<-1ex>[r]_-{F(p_2)} & F(C)
	}
	$$
	commutes. Define $\hat{g} := F(g) \circ i_E$. We will show that $\hat{g}$ is the inverse of $\hat{f}$.  
	
	On one hand, 
	$$
	\hat{g} \circ \hat{f} = F(g) \circ i_E \circ \hat{f} = F(g) \circ F(f) = F(f \circ g) = id_{F(A)}.
	$$
	
	On the other, 
	\begin{align*}
		i_E \circ \hat{f} \circ \hat{g} & = F(f) \circ F(g) \circ i_E = F(g \circ f) \circ i_E = \\
		& = F(p_2 \circ k) \circ i_E = F(k) \circ F(p_2) \circ i_E = \\
		& = F(k) \circ F(p_1) \circ i_E = F(p_1 \circ k) \circ i_E = id_{F(B)} \circ i_E = i_E.
	\end{align*} 
	As $i_E$ is a monomorphism and $i_E \circ \hat{f} \circ \hat{g} = i_E \circ id_E$ we get that $\hat{f} \circ \hat{g} = id_E$.  
\end{pf}

\begin{prop}\label{lemaclau:Bf}
	Let $f: Y \to X$ be an epimorphism in $\CH$.Let $\cF$ be a functor in $\wbSets$. Then the natural map gives a bijection 
	$$
	\cF(B(X)) \cong \eq\left( \xymatrix@1{\cF(B(Y)) \ar@<1ex>[r]^-{\cF Bp_1} \ar@<-1ex>[r]_-{\cF Bp_2} & {\cF(B(Y \times_X Y))}}\right)
	$$
\end{prop}
\begin{pf}
	We only need to show that 
		$$
	\xymatrix@1{
		{B(Y \times_X Y)} \ar@<1ex>[r]^-{Bp_1} \ar@<-1ex>[r]_-{Bp_2} & B(Y) \ar[r]^-{Bf} & B(X).
	}
	$$
	satisfies the conditions of the key lemma \ref{clau}. 
	
	First, as $|f| : |Y| \to |X|$ is surjective, and by the axiom of choice, there exist a section $g: |X| \to |Y|$ such that $f \circ g = \id_{|X|}$. Consider then $\beta g: B(X) \to B(Y)$ which satisfies
	$$
	B(f) \circ \beta g = \beta (|f| \circ g) = id_{B(X)}.
	$$
	
	Now, there exist a unique map $k: |Y| \to |Y| \times_{|X|} |Y|$ such that 
	$$
	\xymatrix{
		|Y| \ar@/_/[ddr]_{{g \circ |f|}} \ar@/^/[drr]^{id} \ar[dr]^k \\
		& {|Y| \times_{|X|} |Y|} \ar[r]_-{|p_1|} \ar[d]_-{|p_2|} & |Y| \ar[d]^-{|f|} \\
		& |Y| \ar[r]_-{|f|} & |X|
	}
	$$
	commutes. Applying the  $\beta$ functor we get a commutative diagram 
	$$
	\xymatrix{
		B(Y) \ar@/_/[ddr]_{{\beta g \circ Bf}} \ar@/^/[drr]^{id} \ar[dr]^{\beta k} \\
		& {B(Y \times_X Y)} \ar[r]_-{B p_1} \ar[d]_-{Bp_2} & B(Y) \ar[d]^-{Bf} \\
		& B(Y) \ar[r]_-{Bf} & B(X)
	}
	$$
as we wanted. 	
\end{pf}

The following proposition shows that the functor $\ex\circ \res$ is naturally isomorphic to the identity.

\begin{prop} \label{exres}
	Let $\cF$ be a functor in $\wbSets$, and let $S$ be a set in $\Sets$. Then the natural map induces a bijection 
	$$
	\cF(\beta S) \cong \hcF(\beta S) = \eq \left( \cF(B(\beta S)) \rightrightarrows \cF(\BB (\beta S)) \right)
	$$
\end{prop}
\begin{proof}
	We will show that the conditions of the key lemma \ref{clau} are verified. Consider
	$$
	\xymatrix@1{
		{\BB (\beta S)} \ar@<1ex>[r] \ar@<-1ex>[r] & {B( \beta S)} \ar[r]^-{{\xi_{\beta S}}} & {\beta S} 
	}
	$$
 Denote by $g$ a section of $\xi_{\beta S}$ by using Lemma \ref{betaex}.

	Moreover, the diagram
	$$
	\xymatrix{
		{B(\beta S)} \ar[d]_-{g \circ \xi_{\beta S}} \ar[r]^-{id} & {B(\beta S)} \ar[d]^-{\xi_{\beta S}} \\
		{B(\beta S)} \ar[r]_-{\xi_{\beta S}} & {\beta S}
	}
	$$
	commutes. Hence, by the Proposition \ref{betafiberproduct}, there exists a continuous map  $k : B(\beta S) \to \BB (\beta S)$ making the following diagram commutative:
	$$
	\xymatrix{
		B(\beta S) \ar@/_/[ddr]_{{g \circ \xi_{\beta S}}} \ar@/^/[drr]^{id} \ar[dr]^{k} \\
		& {\BB (\beta S)} \ar[r] \ar[d] & B(\beta S) \ar[d]^-{\xi_{\beta S}} \\
		& B(\beta S)\ar[r]_-{\xi_{\beta S}} & \beta S
	}
	$$
\end{proof}

From now on we will identify $\cF$ and $\hcF$ when applied to a $\beta$-set, for any $ \cF$ in $\wbSets$.

The following proposition shows that the extension of a functor in $\wbSets$ always verifies the $(\star)$-property (\ref{star}).

\begin{prop}\label{extensionverifiesstar}
	Let $f: Y \to X$ be an epimorphism in $\CH$ and let  $\cF$ be a functor in $\wbSets$. Then 
$$
\hcF(X) \to \hcF (Y) \rightrightarrows \hcF (Y \times_X Y)
$$
is an equalizer. 
\end{prop}

\begin{pf}
Consider the following commutative diagram
$$
\xymatrix{
	{\hcF(X)} \ar[d]_-{\hcF(f)} \ar@{^{(}->}[r]^-{{\cF(\xi_X)}} & \cF(B(X)) \ar@<1ex>[r]^-{{\cF(\tilde{p}_1)}} \ar@<-1ex>[r]_-{{\cF(\tilde{p}_2)}} \ar[d]_-{\cF(Bf)} & {\cF(\BB (X))} \ar[d]^-{{\cF(\BB f)}} \\
	{\hcF(Y)} \ar@<1ex>[d]^-{{\hat{p}_2}} \ar@<-1ex>[d]_-{{\hat{p}_1}} \ar@{^{(}->}[r]^-{{\cF(\xi_Y)}} & \cF(B(Y)) \ar@<1ex>[d]^-{{\cF(B(q_2))}} \ar@<-1ex>[d]_-{{\cF(B(q_1))}} \ar@<1ex>[r]^-{{\cF(\tilde{q}_1)}} \ar@<-1ex>[r]_-{{\cF(\tilde{q}_2)}} & {\cF(\BB (Y))} \ar@<1ex>[d] \ar@<-1ex>[d] \\
	{\hcF(Y \times_X Y)} \ar@{^{(}->}[r]^-{{\hat{i}}} & {\cF(B(Y \times_X Y))} \ar@<1ex>[r] \ar@<-1ex>[r] & {\cF(\BB (Y \times_X Y))}  
}
$$
We denote by $E := \eq \left( \hcF(Y) \rightrightarrows \hcF(Y \times_X Y) \right)$, and the map $\tilde{f}:\hcF(X) \to E$ given by $\hcF(f)$. 

Firstly, we show that $\hcF(f)$ is injective.  By lemma \ref{Bxiandiota}, $\cF( \xi_Y) \circ \hcF(f) = \cF(Bf) \circ \cF(\xi_X)$. Hence, if $a, b \in \hcF(X)$ verify that $\hcF(f)(a) = \hcF(f)(b)$, then 
\begin{multline*}
	(\cF(Bf) \circ \cF(\xi_X))(a) = (\cF(\xi_Y) \circ \hcF(f))(a) =\\ = (\cF(\xi_Y) \circ \hcF(f))(b)   = (\cF(Bf) \circ \cF(\xi_X))(b).	
\end{multline*}
Now, $\cF(Bf) \circ \cF(\xi_X)$ is injective, since $\cF(Bf)$ is injective by Lemma \ref{lemaclau:Bf}, as $F(B(X)) \to F(B(Y)) \rightrightarrows F(B(Y \times_X Y))$ is an equalizer, and $\cF(\xi_X)$ is also injective as $\hcF(X)$ is also an equalizer by construction. Thus, $a=b$ and $\tilde{f}$ is injective. 

To show that $\tilde{f}$ is surjective, consider $c\in E\subset \hcF(Y) \subset \cF(B(Y))$. We are going to construct a preimage $c'\in \hcF(X)$ by constructing a $d \in \cF(B(X))$ such that $\cF(\tilde{p}_1)(d)=\cF(\tilde{p}_2)(d)$ using that the first row of the diagram is an equalizer by Proposition \ref{exres}. 

The existence of a $d \in \cF(B(X))$ such that $\cF(B(f))(d)=\cF(\xi_Y)(c)$ uses that the second column of the diagram is also an equalizer by Proposition \ref{lemaclau:Bf}. As $c$ is in the equalizer of $\hat p_1$ and $\hat p_2$, we have that $\hat p_1(c)=\hat p_2(c)$, and therefore \[\cF(B(q_1))(\cF(\xi_Y)(c))=\hat \iota (\hat p_1(c))=\hat \iota (\hat p_2(c))=\cF(B(q_2))(\cF(\xi_Y)(c)).\]
Hence $\cF(\xi_Y)(c)$ is in the equalizer of $\cF(B(q_1))$ and $\cF(B(q_2))$, proving the existence of such $d$.

The main difficulty now is to show that such $d$ verifies that $\cF(\tilde{p}_1)(d)=\cF(\tilde{p}_2)(d)$. In order to show this, we will construct some useful maps using corollary \ref{proj}. First, as the map $B(f): B(Y)\to B(X)$ is an epimorphism, and in fact it has a section, we can construct a continuos map $u_i:\BB(X)\to B(Y)$ such that $\tilde{p_i}=B(f)\circ u_i$, for $i=1,2$. Therefore, \begin{align*}
\cF(\tilde{p}_i)(d)=\cF(B(f)\circ u_i)(d)& =\cF(u_i)\cF(B(f))(d)\\ &=\cF(u_i)\cF(\xi_Y)(c) =\cF(\xi_Y\circ u_i)(c)    
\end{align*}
for $i=1,2$. 

Using that \[f\circ \xi_Y\circ u_1=\xi_X\circ B(f)\circ u_1= \xi_X \circ\tilde{p_1}=\xi_X \circ\tilde{p_2}=f\circ \xi_Y\circ u_2,\]  there exists a unique continuous map $u: \BB(X)\to Y\times_X Y$ such that $q_i\circ u=\xi_Y\circ u_i$ for $i=1,2$. Using again the projectivity of $\BB(X)$, and the fact that $\xi_{Y\times_XY}$ is an epimorphism, the map $u$ lifts to a map $v:\BB(X)\to B(Y\times_X Y)$ such that $\xi_{Y\times_XY}\circ v=u$. Then, for any $i=1,2$, \[\xi_Y\circ u_i=q_i\circ u=q_i\circ \xi_{Y\times_XY}\circ v=\xi_Y\circ B(q_i)\circ v,\]
and so
\begin{align*} 
	\cF(\xi_Y\circ u_i)(c)& =\cF(\xi_Y\circ B(q_i)\circ v)(c) =\cF( B(q_i)\circ v)(\cF(\xi_Y)(c))\\ & =\cF( B(q_i)\circ v)(\cF(B(f))(d))=\cF(B(f)\circ B(q_i)\circ v)(d).
\end{align*}
But $q_1\circ f=q_2\circ f$ by definition of the fibre product, so this last element is independent of $i$. Consequently, $\tilde{f}$ is surjective.  
\end{pf}

\begin{prop}\label{resex}
	Let $\cF$ be a functor in $\wCH_{\star}$, and let $X$ be a compact Hausdorff space. Then the natural map gives a bijection
	$$
	F(X) \cong \hcF (X) = \eq \left( \cF(B(X)) \rightrightarrows \cF(\BB (X))\right).
	$$
\end{prop}
\begin{pf}
	Consider the diagram 
	$$
	\xymatrix@1{
		{\BB (X)} \ar[r]^-{{\xi_{\widetilde{B}(X)}}} & {B(X) \times_X B(X)} \ar@<1ex>[r]^-{p_1} \ar@<-1ex>[r]_-{p_2} & B(X) \ar[r]^-{\xi_X} & X
	}.
	$$
	
	We denote by $E$ the equalizer $\eq \left( \cF(B(X)) \rightrightarrows \cF(B(X) \times_X B(X)) \right)$. Since $\cF$ verifies property $(\star)$, we have that $\cF(X) \cong E$ through the map $\cF(\xi_X): \cF(X) \to \cF(B(X))$. 
	
	As $\xi_{B(X) \times_X B(X)} : \BB (X) \to B(X) \times_X B(X)$ is an epimorphism, 	
	$$
	\xymatrix@1{
		{\cF(B(X) \times_X B(X))} \ar[r]^-{{\cF(\xi_{\widetilde{B}(X)})}} & {\cF(\BB (X))} \ar@<1ex>[r] \ar@<-1ex>[r] & {\cF(\BB (X) \times_{\widetilde{B}(X)} \BB(X))} 
	}
	$$
	is an equalizer, again because $\cF$ verifies property $(\star)$. In particular, $\cF(\xi_{\widetilde{B}(X)})$ is injective. 
	
	From the following commutative diagram
	$$
	\xymatrix@1{
		\cF(X) \ar[r]^-{\cF(\xi_X)} & \cF(B(X)) \ar@<1ex>[r]^-{\cF(p_1)} \ar@<-1ex>[r]_-{\cF(p_2)} & {\cF(B(X) \times_X B(X))} \ar[r]^-{\cF(\xi_{\widetilde{B}(X)})} & {\cF(\BB (X))}
	}
	$$
	we deduce that $\hcF(X) \subseteq E$ and we are done. 
\end{pf}

\begin{proof}[Proof of Theorem \ref{equivalence}]
In order to prove the theorem, we will see that the functors $\res$ and $\ex$ are well-defined and are inverses one of the other. The restriction functor is clearly well-defined, and the $\ex$ functor in lemma \ref{extension} lands inside $\wCH_{\star}$ because of proposition \ref{extensionverifiesstar}. 

Now, Propositions \ref{resex} and \ref{exres} can be used to show that $\ex\circ \res$ and $\res\circ \ex$ are isomorphic to their respective identity functors, which shows the result.
\end{proof}

\begin{proof}[Proof of Corollary \ref{equivalencestar}]
By using the proof of theorem \ref{equivalence}, we only need to show that, if $\cF $ is in $\wbSets$ and it sends coproducts to products, then $\hcF$ also sends coproducts to products. But corollary \ref{coproductfreeresolutions} shows that the free resolution of the coproduct of two compact Hausdorff spaces $X$ and $Y$ is 
\[
	\xymatrix{
		{\BB (X) \sqcup \BB (Y)} \ar@<1ex>[r]^-{{\tilde{p_1} \sqcup \tilde{q_1}}}  \ar@<-1ex>[r]_-{{\tilde{p_2} \sqcup \tilde{q_2}}} & {B(X) \sqcup B(Y)} \ar[r]_-{{\xi_X \sqcup \xi_Y}} & {X \sqcup Y} 
	}
\]
hence 
\[\hcF(X\sqcup Y):=\eq\left( 	\xymatrix{
		{\cF(B (X) \sqcup B (Y))} \ar@<1ex>[r]^-{{\cF(\tilde{p_1} \sqcup \tilde{q_1})}}  \ar@<-1ex>[r]_-{{\cF(\tilde{p_2} \sqcup \tilde{q_2})}} & {\cF(\BB(X) \sqcup \BB(Y))}}\right)\]
Using that  $\cF $ sends coproducts to products we get that 
\begin{multline*}
    \hcF(X\sqcup Y):=\\ =\eq\left( 	
			\xymatrix{ {\cF(B (X))\times \cF( B (Y))}
\ar@<1ex>[r]^-{{\cF(\tilde{p_1}) \times \cF(\tilde{q_1})}}  \ar@<-1ex>[r]_-{{\cF(\tilde{p_2}) \times \cF(\tilde{q_2})}} & {\cF(\BB(X))\times \cF(\BB(Y))} } 
		\right)= \\ = \hcF(X)\times \hcF(Y)
\end{multline*}

\end{proof}

\section{A construction of the condensed set associated to any presheaf on compact Hausdorff spaces}

Given any contravariant functor $\cF:\CH\to \Sets$, the general construction of the sheaf associated to a presheaf determines a condensed set $\cF^{\sharp}$; it gives a functor $\ ^{\sharp}:\wCH\to \cSets$, left adjoint to the inclusion $\cSets\to \wCH$. 

Firstly, observe that our result can be used to construct the condensed set associated to any functor $\cF$ in $\wCH_{\times}$. 

\begin{cor}
    Let $\cF$ be a functor from $\CH$ to $\Sets$ that sends finite coproducts to products. Then $\ex(\res(\cF))=\hcF$ is a condensed set. 
\end{cor}

If we have a contravariant functor $\cF:\CH\to \Sets$, we can restrict it to $\bSets$ and apply in this category the general construction of a functor preserving products, via the functor $\ ^+:\wbSets\to \wbSets_{\times}$ adjoint to the inclusion. It is a particular case of the plus construction for constructing the sheaf associated to a presheaf, but we only have to apply it once instead of twice. 

Given a set $S$, consider the following category $\cU_S$, whose objects $U$ are finite partitions of $S$, so $U=\{U_i\}_{i\in I}$ for $I$ a finite set, $U_i\subset S$, $U_i\ne \emptyset$ unless $S=\emptyset$, and such that $S=\bigsqcup_{i\in I} U_i$. Maps are given by refinements, so if $V=\{V_j\}_{j\in J}$ is another partition, a morphism $\psi: V\to U$ corresponds to a  map $\psi:J\to I$ such that for every $i\in I$ we have \[U_i=\bigsqcup_{j\in \psi^{-1}(i)} V_j\] (so there is at most one map, hence $\cU_S$ is a poset). Notice that we always have an inclusion $\iota_j:V_j\hookrightarrow U_{\psi(j)}$ for every $j\in J$. 

The category $\cU_S$ is cofiltered, as it is not empty, since $\{S\}$ is an object in $\cU_S$, and any two coverings $U=\{U_i\}_{i\in I}$ and $V=\{V_j\}_{j\in J}$ have a common refinement $W$, given by the intersections $W_{ij}=U_i\cap V_j$ (if $U_i\cap V_j \ne \emptyset$).  

For any functor $\cF$ in $\wbSets$, consider the contravariant functor $\cF_S:\cU_S\to \Sets$ given by $\cF_S(U)=\prod_{i\in I} \cF(\beta U_i)$ if $U=\{U_i\}_{i\in I}$, and if $\psi: V:=\{V_j\}_{j\in J}\to U$ is a morphism, the map $\cF_S(\psi):\cF_S(U)\to \cF_S(V)$ is given by \[\prod_{i\in I} \upsilon_i:\prod_{i\in I} \cF(\beta U_i)\to \prod_{i\in I} \prod_{j\in \psi^{-1}(i)} \cF(\beta V_j) \] the product of the maps  $\upsilon_i:\cF(\beta U_i)\to \prod_{j\in \psi^{-1}(i)} \cF(\beta V_j)$ given by the universal property from the maps  
$F(\iota_j):\cF(\beta U_{\psi(j)})\to \cF(\beta V_j)$ where $\iota_j:V_j\hookrightarrow U_{\psi(j)}$ is the inclusion. We define $\cF^+(S):=\colim \cF_S$. 

Notice that we have a map $\cF(\beta S)\to \cF^+(\beta S)$ as $S$ is one object of $\cU_S$, in fact a terminal object. In fact, two elements in $\cF(\beta S)$ give the same element in $\cF^+(\beta S)$  if there exists a covering $U=(U_i)_{i\in I}$ of $S$ and the elements have the same image by the map $\cF(S)\to \prod_{i\in I}\cF(U_i)$. 

In general, an element in $\cF^+(\beta S)$ is an element in  $\prod_{i\in I} \cF(\beta U_i)$ for some covering $U=(U_i)_{i\in I}$ of $S$; and, given two elements $s_1\in\cF_S(U)$ and $s_2\in \cF_S(V)$, for $U$ and $V$ objects in $\cU_S$, they are equal in $\cF^+(\beta S)$ if there is a common refinement  $\psi:W\to U$ and $\phi:W\to V$ such that  $\cF_S(\psi)(s_1)=\cF_S(\phi)(s_2)$.

One easily sees that the assignment can be promoted to a functor, as given any map $f:\cF\to \cG$ in $\wbSets$, and given any $S$ in $\Sets$ and a covering $U=(U_i)_{i\in I}$, we have maps $f(\beta U_i):\cF(\beta U_i)\to \cG(\beta U_i)$, and hence maps $ f(U): \cF_S(U)\to \cG_S(U)$ which commute with the restriction maps $\cF_S(\psi)$ for any $\psi:U\to V$ in $\cU_S$, so a map $f^+:\cF^+(\beta S)\to \cG^+(\beta S)$.

\begin{lem}
For any functor $\cF$ in $\wbSets$, the functor $\cF^+$ sends finite coproducts to finite products, so it is in fact in $\wbSets_{\times}$. 
\end{lem}
\begin{pf} Observe first that $\cF^+(\emptyset)=\{*\}$, a final object in $\Sets$. Notice that the  $\cU_{\emptyset}$ has two objects, the trivial covering $\emptyset=\emptyset$, and the empty covering, that has no subsets. The empty covering is a refinement of the trivial covering. Now, the product over the empty set is a terminal set, and two elements of $\cF(\emptyset)$ become the same element on the empty covering.

Now, suppose that $\beta S=\beta T_1\sqcup \beta T_2\cong \beta(T_1\sqcup T_2)$, this last equality by Proposition \ref{BSetsCat}. But lemma \ref{isobeta} shows that there exists a bijection as sets $S\cong T_1\sqcup T_2$, hence there exists $U_1$ and $U_2$ disjoint subsets of $S$ such that $T_i\cong U_i$ for $i=1,2$, giving the bijection $S\cong T_1\sqcup T_2$. Hence, to show that  $\cF^+(\beta S)\cong \cF^+(\beta T_1)\times \cF^+(\beta T_2)$ it is sufficient to show it when $T_i$ are subsets of $S$ covering $S$ and the isomorphism is given by the equality.

Hence, consider now two disjoint subsets $U_1$ and $U_2$ of $S$ such that $S=U_1\sqcup U_2$. We want to show that the natural map
\[\rho:\cF^+(\beta S)\to \cF^+(\beta U_1)\times \cF^+(\beta U_2)\]
given by the maps $U_i\to U_1\sqcup U_2$, is a bijection.

Suppose first that $s_1,s_2\in \cF^+(S)$ goes to the same $s$ by $\rho$; as $\{U_i\}_{i\in \{1,2\}}$ is a covering of $S$, this means it appears in $\cU_{S}$. Now, being $s$ in $\cF^+(\beta U_i)$ it means there are coverings $\{U_{j,i}\}_{j\in I_j}$ of $\beta U_i$ such that $s$ goes to $(\sigma_{j,i})$ in $\prod_{i=1}^2 \prod_{j\in I_j} \cF(\beta U_{j,i})$. But $\{U_{j,i}\}_{j\in I_j,i\in \{1,2\}}$ is a covering in $\cU_S$, hence $s_1$ and $s_2$ are equal in $\cF^+(\beta S)$, hence they are equal. Therefore, $\rho$ is injective. 

Now, suppose we have $(s_1,s_2)\in  \cF^+(\beta U_1)\times \cF^+(\beta U_2)$. This means that  there are coverings $\{U_{j,i}\}_{j\in I_j}$  and we have $(\sigma_{j,i})$ in $\prod_{i=1}^2 \prod_{j\in I_j} \cF(\beta U_{j,i})$ such that $s_i\mapsto  (\sigma_{j,i})\in \prod_{j\in I_j} \cF(\beta U_{j,i})$ by the map that goes from the colimit to one of the objects of the colimit. But then $(\sigma_{j,i})$ determines an element $s\in \cF^+(\beta S)$. Clearly $\rho(s)=(s_1,s_2)$, hence $\rho$ is a bijection and $\cF^+$ is in $\wbSets_{\times}$.
\end{pf}

Finally, it is straightforward to prove that the given construction is the left adjoint to the inclusion $\wbSets_{\times}\to \wbSets$.

\begin{cor}
    The composition of the functors \[\wCH\overset{\res}{\to}\wbSets\overset{\ ^+}{\to}\wbSets_{\times}\overset{\ex}{\to} \cSets\]
    is isomorphic to the $\sharp$-construction associating a sheaf to a presheaf. 
\end{cor}

\end{document}